\documentclass[a4paper,12pt, leqno]{amsart}
\usepackage{amsmath,amsthm,amsfonts,amssymb,graphicx,mathrsfs,esint,longtable,mathtools, todonotes}
\usepackage[colorlinks=true, linkcolor=red, urlcolor=red, citecolor=green]{hyperref}
\usepackage{cleveref}
\usepackage{color}

\newcommand{\R}{{\mathbb R}}

\def\Xint#1{\mathchoice
{\XXint\displaystyle\textstyle{#1}}%
{\XXint\textstyle\scriptstyle{#1}}%
{\XXint\scriptstyle\scriptscriptstyle{#1}}%
{\XXint\scriptscriptstyle\scriptscriptstyle{#1}}%
\!\int}
\def\XXint#1#2#3{{\setbox0=\hbox{$#1{#2#3}{\int}$ }
\vcenter{\hbox{$#2#3$ }}\kern-.6\wd0}}

\def\dashint{\Xint-}

\newcommand{\mres}{\mathbin{\vrule height 1.6ex depth 0pt width
0.13ex\vrule height 0.13ex depth 0pt width 1.3ex}}

\newcommand{\N}{\mathbb{N}}

\newcommand{\Z}{\mathbb{Z}}

\newcommand{\diam}{\textnormal{diam$\,$}}

\newcommand{\Ha}{\mathcal{H}}
\newcommand{\modd}{\textnormal{mod}}
\newcommand{\capp}{\textnormal{cap}}

\newcommand{\dist}{\textnormal{dist}}

\newtheorem{theorem}{\textbf{THEOREM}}[section]
\newtheorem{lemma}[theorem]{\textsc{Lemma}}
\newtheorem{proposition}[theorem]{\textsc{Proposition}}
\newtheorem{corollary}[theorem]{\textsc{Corollary}}
\theoremstyle{definition}
{\theoremstyle{remark} \newtheorem{remark}[theorem]{Remark}}
\def\charfn_#1{{\raise1.2pt\hbox{$\chi_{\kern-1pt\lower3pt\hbox{{$\scriptstyle#1$}}}$}}}
\def\leq{\leqslant }
\def\geq{\geqslant }

\def\Xint#1{\mathchoice
{\XXint\displaystyle\textstyle{#1}}%
{\XXint\textstyle\scriptstyle{#1}}%
{\XXint\scriptstyle\scriptscriptstyle{#1}}%
{\XXint\scriptscriptstyle\scriptscriptstyle{#1}}%
\!\int}
\def\XXint#1#2#3{{\setbox0=\hbox{$#1{#2#3}{\int}$}
\vcenter{\hbox{$#2#3$}}\kern-.5\wd0}}

\def\dashint{\Xint-}

\begin{document}

\title{Duality of moduli in regular metric spaces}
\author{Atte Lohvansuu and Kai Rajala} 
\let\thefootnote\relax\footnote{\emph{Mathematics Subject Classification 2010:} Primary 30L10, Secondary 30C65, 28A75,
51F99.}
\thanks{Both authors were supported by the Academy of Finland, project number 308659. The first author was also supported by the Vilho, Yrj\"o and Kalle V\"ais\"al\"a foundation. Parts of this research were carried out when the first author was visiting the Mathematics Department in the University of Michigan. He would like to thank the department for hospitality. }
\begin{abstract}Gehring \cite{Gehring1962} and Ziemer \cite{Ziemer1967} proved that the $p$-modulus of the family of paths connecting two continua is dual to the $p^*$-modulus of the corresponding family of separating hypersurfaces. In this paper we show that a similar result holds in complete Ahlfors-regular metric spaces that support a weak $1$-Poincar\'e inequality. As an application we obtain a new characterization for quasiconformal mappings between such spaces. 
\end{abstract}

\maketitle

\renewcommand{\baselinestretch}{1.2}

\section{Introduction}
The modulus of a path family is a widely used tool in geometric function theory and its generalizations to $\R^n$ and furthermore to metric spaces, see \cite{HKST},\cite{LehtoVirtanenQC} and \cite{RickmanQR}.

Given $1\leq p<\infty$ and a family $\Gamma$ of paths in a metric measure space $(X, d, \mu)$, the $p$-modulus of $\Gamma$ is defined to be 
\[
\modd_p\Gamma:=\inf_\rho\int_X\rho^p\ d\mu,
\]
where the infimum is taken over all admissible functions of $\Gamma$, i.e., Borel measurable functions $\rho: X\rightarrow [0, \infty]$ that satisfy
\[
\int_\gamma\rho\ ds\geq 1
\]
for all locally rectifiable $\gamma\in\Gamma$. If no admissible functions exist, the modulus is defined to be $\infty$.
The definition of modulus can be generalized considerably, as was done by Fuglede in his 1957 paper \cite{Fuglede1957}. For example, instead of paths we can consider surfaces by defining the modulus with exactly the same formula but requiring the admissible functions to satisfy
\[
\int_S\rho\ d\sigma_S\geq 1
\]
for all surfaces $S$ in the family. Here $\sigma_S$ denotes some Borel-measure associated to $S$. In our applications $\sigma_S$ will be comparable to a Hausdorff measure restricted to $S$. 

Our main result is concerned with Ahlfors $Q$-regular complete metric spaces that support a weak 1-Poincar\'e inequality. We also assume $Q>1$. See Section \ref{section:definitions} for all relevant definitions. Fix such a metric measure space $(X, d, \mu)$. Given a domain $G\subset\subset X$ and disjoint nondegenerate continua $E, F\subset G$ we denote by $\Gamma(E, F; G)$ the family of rectifiable paths in $G$ that join $E$ and $F$. Similarly, we denote by $\Gamma^*(E, F; G)$ the family of compact sets $S\subset \overline G$ that have finite $(Q-1)$-dimensional Hausdorff measure and separate $E$ and $F$ in $G$. By separation we mean that $E$ and $F$ belong to disjoint components of $G-S$. We equip each surface $S$ with the restriction of the $(Q-1)$-dimensional Hausdorff measure on $S\cap G$. For $1<p<\infty$, denote $p^*=\frac{p}{p-1}$.

The main purpose of this paper is to prove the following connection between the path modulus and the modulus of separating surfaces.
\begin{theorem}\label{thm:duality}
Let $1<p<\infty$. There is a constant $C$ that depends only on the data of $X$ such that 
\begin{equation}\label{eq:dualityintro}
\frac{1}{C}\leq\modd_p\Gamma(E, F; G)^\frac{1}{p} \cdot \modd_{p^*}\Gamma^*(E, F; G)^\frac{1}{p^*}\leq C,
\end{equation}
for any choice of $E, F$ and $G$. Here it is understood that $0\cdot\infty=1$.
\end{theorem}
Gehring \cite{Gehring1962} and Ziemer \cite{Ziemer1967} proved that \eqref{eq:dualityintro} holds in $\R^n$ with $C=1$.

As an application of Theorem \ref{thm:duality} we find a new characterisation for quasiconformal maps between regular spaces. Let $Y$ be another complete Ahlfors $Q$-regular space that supports a weak 1-Poincar\'e inequality. Recall that a homeomorphism $f: X\rightarrow Y$ is (geometrically) $K$-\emph{quasiconformal} if there exists a constant $K\geq 1$ such that for every family $\Gamma$ of paths in $X$
\begin{equation}\label{eq:pathintro}
\frac{1}{K}\modd_Q(f\Gamma)\leq\modd_Q\Gamma\leq K\modd_Q(f\Gamma).
\end{equation}
Here $f\Gamma=\{f\circ\gamma\ |\ \gamma\in\Gamma\}$. 

\begin{corollary}\label{cor:intro}
Let $X$ and $Y$ be as above. A homeomorphism $f: X\rightarrow Y$ is $K$-quasiconformal if and only if there is a constant $C$, such that
\[
\frac{1}{C}\modd_{Q^*}\Gamma^*(E, F; G)\leq\modd_{Q^*}\Gamma^*(fE, fF; fG)\leq C\modd_{Q^*}\Gamma^*(E, F; G)
\]
for all $E, F$ and $G$ as above. The constants $C$ and $K$ depend only on each other and the data of $X$ and $Y$. 
\end{corollary}
See Section \ref{section:results} for the proof. We remark that the ``only if" part follows also from the recent work of Jones, Lahti and Shanmugalingam \cite{JonesLahtiShan2018}.

This paper is organized as follows: In Section \ref{section:definitions} we introduce the main tools for later use. In Section \ref{section:results} we state our main results, Theorems \ref{thm:lowerbound} and \ref{thm:upperbound}, which are more general versions of the lower and upper bounds in \eqref{eq:dualityintro}. We also show how these results imply Corollary \ref{cor:intro}. Theorem \ref{thm:lowerbound} is proved in Section \ref{section:lowerbound} along the lines of \cite{Gehring1962} and \cite{Ziemer1967}, applying coarea estimates. The proof of Theorem \ref{thm:upperbound}, which seems to be new even in the euclidean setting, is given in Section \ref{section:upperbound}. Section \ref{section:examples} contains an example showing the necessity of the $1$-Poincar\'e inequality in Theorem \ref{thm:duality}. 

\vskip 10pt
\noindent
{\bf Acknowledgement.}
We are grateful to Panu Lahti and Nageswari Shanmugalingam for pointing out an error in an earlier version of the manuscript. 



\section{Preliminaries}\label{section:definitions}
\subsection{Doubling measures}
A Borel-regular measure $\mu$ is called \emph{doubling} with doubling constant $C_\mu>1$ if 
\begin{equation}\label{eq:doubling}
0< \mu(2B)\leq C_\mu\mu(B) <\infty 
\end{equation}
for all balls $B\subset X$. Iterating (\ref{eq:doubling}) shows that there are constants $C_\mu'$ and $s>0$ that depend only on $C_\mu$ such that for any $x, y\in X$ and $0<r\leq R<
\diam(X)$ with $x\in B(y, R)$,
\begin{equation}\label{eq:doubling2}
\frac{\mu(B(y, R))}{\mu(B(x, r))}\leq C'_\mu\left(\frac{R}{r}\right)^s.
\end{equation}
In fact, we can choose $s\geq \log_2C_\mu$.

The space $X$ is said to be \emph{Ahlfors $Q$-regular}, or just $Q$-regular, if there are constants $a$ and $A>0$ such that 
\begin{equation}\label{eq:ahlforsreg}
ar^Q\leq\mu(B(x, r))\leq Ar^Q
\end{equation}
for every $x\in X$ and $0<r<\diam(X)$. It follows immediately from the definitions that $Q$-regular spaces are doubling.

\subsection{Moduli}
Let $\mathscr{M}$ be a set of Borel-regular measures on $X$ and let $1\leq p<\infty$. We define the $p$-\emph{modulus} of $\mathscr{M}$ to be
\[
\modd_p\mathscr{M}=\inf\int_X\rho^p\ d\mu,
\]
where the infimum is taken over all Borel measurable functions $\rho: X\rightarrow [0, \infty]$ with 
\begin{equation}\label{eq:modulusehto}
\int_X\rho\ d\nu\geq 1
\end{equation}
for all $\nu\in \mathscr{M}$. Such functions are called \emph{admissible functions of $\mathscr{M}$}. If there are no admissible functions we define the modulus to be infinite. If $\rho$ is an admissible function for $\mathscr{M}-\mathscr{N}$ where $\mathscr{N}$ has zero $p$-modulus, we say that $\rho$ is \emph{$p$-weakly admissible} for $\mathscr{M}$. As a direct consequence of the definitions we see that the $p$-modulus does not change if the infimum is taken over all $p$-weakly admissible functions. If some property holds for all $\nu\in \mathscr{M}-\mathscr{N}$ we say that it holds for $p$-\emph{almost every} $\nu$ in $\mathscr{M}$.

We can also use paths instead of measures; if $\Gamma$ is a family of locally rectifiable paths in $X$ we define the \emph{path p-modulus} of $\Gamma$ as before with
\[
\modd_p\Gamma=\inf\int_X\rho^p\ d\mu,
\]
but require that 
\[
\int_\gamma\rho\ ds\geq 1
\]
for every locally rectifiable $\gamma\in\Gamma$. See \cite{Vaisala} or \cite{GMT} for the definition and properties of path integrals over locally rectifiable paths. Most of the path families considered in this paper will be of the form
\[
\Gamma(E, F; G):=\{\text{paths that join } E \text{ and } F \text { in } G \},
\]
where $E, F\subset G$ are disjoint continua and $G$ is a domain in $X$. The modulus of $\Gamma(E, F; G)$ does not change if we consider only injective paths, see \cite[Proposition 15.1]{Semmes1996b}. For injective paths 
\[
\int_\gamma\rho\ ds = \int_{|\gamma|}\rho\ d\Ha^1,
\]
as can be seen from the area formula \cite[2.10.13]{GMT}. This implies that the modulus of any subfamily of $A$ of $\Gamma(E, F; G)$ is the same as the modulus of the family
\[
\{ \Ha^1\mres |\gamma|\ |\ \gamma\in A\},
\]
so in this sense the two definitions of the modulus are equal. 

We will need the following basic lemma in multiple occasions. It is a combination of the lemmas of Fuglede and Mazur, see \cite[p. 19, 131]{HKST}.
\begin{lemma}\label{lemma:minimizer}
Let $\mathscr{M}$ be a set of Borel measures on $X$ and $1<p<\infty$. Suppose $\modd_p\mathscr{M}<\infty$. Then there is a sequence $(\rho_i)_{i=1}^\infty$ of admissible functions of $\mathscr{M}$ that converges in $L^p(X)$ to a $p$-weakly admissible function $\rho$ of $\mathscr{M}$ such that for $p$-almost every $\nu\in \mathscr{M}$
\begin{equation}\label{eq:mini3}
\int_X\rho_i\ d\nu\rightarrow\int_X\rho\ d\nu<\infty
\end{equation}
and
\begin{equation}\label{eq:mini2}
\modd_p\mathscr{M}=\int_X\rho^p\ d\mu.
\end{equation}
\end{lemma}
\begin{remark}
Lemma \ref{lemma:minimizer} holds for the path modulus of a path family $\Gamma$ with the obvious modification of replacing (\ref{eq:mini3}) with 
\[
\int_\gamma\rho_i\ ds\rightarrow\int_\gamma\rho\ ds<\infty
\]
for all $\gamma\in\Gamma$.
\end{remark}

\subsection{Upper gradients}
A Borel function $\rho: X\rightarrow [0, \infty]$ is an \emph{upper gradient} of a function $u: X\rightarrow\overline\R$, if
\begin{equation}\label{eq:ug}
|u(\gamma(1))-u(\gamma(0))|\leq \int_\gamma\rho\ ds
\end{equation}
for all rectifiable paths $\gamma: [0, 1]\rightarrow X$. If $|u(\gamma(0))|$ or $|u(\gamma(1))|$ equal $\infty$, we agree that the left side of (\ref{eq:ug}) equals $\infty$. If (\ref{eq:ug}) fails only for a family of paths of zero $p$-modulus, we say that $\rho$ is a $p$-\emph{weak} upper gradient. 
The following lemma will be useful in the sequel, and will be used without further mention. It allows the use of weak upper gradients in place of upper gradients in all the relevant results used in this paper. This is Proposition 6.2.2 of \cite{HKST}. 
\begin{lemma}
If $u: X\rightarrow \R$ has a $p$-weak upper gradient $\rho\in L^p(X)$ in $X$, then there is a decreasing sequence $(\rho_k)_{k=1}^\infty$ of upper gradients of $u$ that converges to $\rho$ in $L^p(X)$.
\end{lemma}

\subsection{Maximal functions}
Suppose $\mu$ is doubling and $R>0$. The restricted Hardy-Littlewood maximal function $\mathcal{M}_Ru$ of an integrable function $u: X\rightarrow \overline\R$ is defined as
\[
\mathcal{M}_Ru(x)=\sup_{0<r\leq R}\dashint_{B(x, r)}|u|\ d\mu,
\]
where 
\[
\dashint_{B}v\ d\mu := \frac{1}{\mu(B)}\int_Bv\ d\mu.
\]
The Hardy-Littlewood maximal function $\mathcal{M}u$ can then be defined as 
\[
\mathcal{M}u=\sup_{R>0}\mathcal{M}_Ru.
\]
In doubling spaces $\mathcal{M}u$ is Borel measurable whenever $u$ is, and the assignment $u\mapsto \mathcal{M}u$ defines a bounded operator $L^p(X)\rightarrow L^p(X)$ for any $1<p<\infty$, with bound depending only $p$ and the doubling constant of $X$, see \cite[Chapter 3.5]{HKST} for details.

\subsection{Codimension 1 spherical Hausdorff measure}
Given a Borel-regular measure $\mu$, the \emph{codimension 1 spherical Hausdorff $\delta$-content} of a set $A\subset X$ is defined as 
\[
\Ha_\delta(A):=\inf\sum_i\frac{\mu(B_i)}{r_i},
\]
where the infimum is taken over countable covers $\{B_i\}$ of $A$, and $B_i=B(x_i, r_i)$ for some $x_i\in X$ and $r_i\leq\delta$. The codimension 1 spherical Hausdorff measure of $A$ is then defined to be 
\[
\Ha(A):=\sup_{\delta>0}\Ha_\delta(A).
\]
By the Carath\'eodory construction $\Ha$ is also a Borel-regular measure. If $X$ is $Q$-regular, $Q \geq 1$, and $\mu$ the $Q$-dimensional Hausdorff measure, then $\Ha$ is comparable to the $(Q-1)$-dimensional Hausdorff measure. 
\subsection{Poincar\'e inequalities}
The space $X$ is said to support a \emph{weak} $p$-\emph{Poincar\'e inequality} with constants $C_P$ and $\lambda_P$ if all balls in $X$ have positive and finite measure, and
\[
\dashint_B|u-u_B|\ d\mu\leq C_P\diam(B)\left(\dashint_{\lambda_PB}\rho^p\ d\mu\right)^\frac{1}{p}
\]
for all functions $u\in L^1_{loc}(X)$ and all upper gradients $\rho$ of $u$.

In the sequel we will encounter function-upper gradient pairs $(v, \rho_v)$ that are defined only on some open and connected set $G\subset X$. For such pairs the Poincar\'e inequality can be applied on any ball $B$ with $\lambda_PB\subset G$, or $B\subset\subset G$ if $\lambda_P=1$. To see this, let $c>1$ be such that $cB\subset G$ and replace $v$ with $v'=v\chi_{cB}$ and $\rho_v$ with $\rho'=\rho_v\chi_{B}+\infty\chi_{X-B}$. Then $\rho'$ is an upper gradient of $v'$ and $v'$ is locally integrable on $X$.
\subsection{Whitney-type coverings}
We will need the following modification of Lemma 4.1.15 in \cite{HKST} in multiple occasions. Here we assume that $(X, d, \mu)$ is a doubling metric measure space, $\Omega\subset X$ is open and bounded and $X-\Omega$ is nonempty.
\begin{lemma}\label{lemma:whitneypeite}
Given any subset $A\subset \Omega$ and natural number $n\geq 2$, there exists a countable collection $\mathcal{B}=\{B(x_i, r_i)\}$ of balls in $\Omega$, such that 
\begin{enumerate}
\item[(i)] $x_i\in A$ and $r_i=\frac{1}{2n}d(x_i, X-\Omega)$ for all $i$
\item[(ii)] If $B_i, B_j\in\mathcal{B}$ intersect, then 
\[
\frac{1}{2}\leq \frac{r_i}{r_j}\leq 2
\]
\item[(iii)] For all $x\in \Omega$
\[
\chi_A(x)\leq \sum_{B\in \mathcal{B}}\chi_{2B}(x)\leq C,
\]
where $C$ depends only on the doubling constant of $\mu$.
\end{enumerate}
\end{lemma}
\begin{proof}
Let $A\subset \Omega$ and $2\leq n\in \N$. Denote $d(x)=d(x, X-\Omega)$. For any $k\in\Z$ let
\[
A_k=\lbrace x\in A\ |\ 2^{k-1}<d(x)\leq 2^k\rbrace
\]
and
\[
\mathcal{F}_k=\lbrace B(x, d(x)/10n)\ |\ x\in A_k\rbrace.
\]
Apply the $5r$-covering theorem on $\mathcal{F}_k$ to find a countable pairwise disjoint collection $\mathcal{G}_k\subset \mathcal{F}_k$ such that
\[
\bigcup_{B\in \mathcal{F}_k}B\subset \bigcup_{B\in\mathcal{G}_k}5B.
\]
Denote by $\mathcal{B}$ the collection of all balls $5B$ with $B\in\mathcal{G}_k$ for some $k\in\Z$. Then $\mathcal{B}$ is countable and (i) is satisfied. A simple application of the triangle inequality proves (ii). The lower bound of (iii) follows from the definition of $\mathcal{B}$. Let $x\in\Omega$. By (i) and (ii) there is a $k\in\Z$ such that balls $B\in\mathcal{B}$ whose scaled versions $2B$ contain $x$ must come from either $\mathcal{G}_k$ or $\mathcal{G}_{k-1}$. Now let $10B_1, \ldots, 10B_N$ be balls arising from $\mathcal{G}_k$ that contain $x$ with radii $r_1, \ldots, r_N$ respectively, so that $r_1\geq r_i$ for all $i=1, \ldots, N$. By the definition of $\mathcal{G}_k$ the balls $B_i$ are disjoint, so by the doubling property and (ii) 
\[
\mu(11B_1)\geq \sum_{i=1}^N\mu(B_i)\geq CN\mu(11B_1),
\]
where $C$ depends only on the doubling constant of $\mu$. The same argument can be applied to $\mathcal{G}_{k-1}$ and (iii) follows.
\end{proof}
\section{Main results}\label{section:results}
Assume for the rest of the text that $(X, d, \mu)$ is a complete metric measure space that supports a weak $1$-Poincar\'e inequality with constants $C_P$ and $\lambda_P$. Assume also that $\mu$ is Borel-regular and doubling so that it satisfies (\ref{eq:doubling2}) with some $C_\mu$ and $s>1$. Note that the doubling condition implies that $X$ is proper and therefore also separable. By \cite[Part I, II.3.11]{Schwartz1973} $\mu$ is in fact a Radon-measure. 

Fix a domain $G\subset\subset X$ and two disjoint nondegenerate continua $E, F\subset G$. Denote $G'=G-(E\cup F)$. Denote by $\Gamma$ the set of all injective rectifiable paths $\gamma: [0, 1]\rightarrow G$ with $\gamma(0)\in E$ and $\gamma(1)\in F$. For any $1\leq p<\infty$ denote
\begin{equation}\label{eq:polkumod}
\modd_p\Gamma:=\modd_p\lbrace\Ha^1\mres |\gamma| \ |\ \gamma\in\Gamma\rbrace.
\end{equation}
Similarly, denote by $\Gamma^*$ the set of all compact subsets $S\subset \overline G$ that separate $E$ and $F$ in $G$ and have finite $\Ha$-measure in $G$. Abbreviate
\begin{equation}\label{eq:pintamod}
\modd_q\Gamma^*=\modd_q\lbrace \Ha\mres S\cap G\ |\ S\in \Gamma^*\rbrace. 
\end{equation}
The requirement $\Ha(S\cap G)<\infty$ is redundant since the modulus of the family of sets with infinite $\Ha$-measure is zero. Nevertheless we prefer to work with sets of finite $\Ha$-measure. 

We denote $C=C(X)$ if some constant $C>0$ depends only on the data of $X$, i.e., the constants $s, C_\mu, C_P$ and $\lambda_P$. The same symbol $C$ will be used for various different constants. Denote $p^*=\frac{p}{p-1}$ for each $1<p<\infty$. The main results of this paper are the following: 
\begin{theorem}\label{thm:lowerbound}
Let $1<p<\infty$. If $\modd_p\Gamma \neq 0$, then 
\[
C\leq(\modd_p\Gamma)^\frac{1}{p}(\modd_{p^*}\Gamma^*)^\frac{1}{p^*}, 
\]
where the constant $C$ depends only on the data of $X$. If $\modd_p\Gamma =0$, then ${\modd_{p^*}\Gamma^*=\infty}$.
\end{theorem}
\begin{theorem}\label{thm:upperbound}
Let $1<p<\infty$. If $\modd_{p^*}\Gamma^*< \infty$, then 
\begin{equation}
\label{eq:mobido}
(\modd_p\Gamma)^\frac{1}{p}(\modd_{p^*}\Gamma^*)^\frac{1}{p^*}\leq C, 
\end{equation}
where the constant $C$ depends only on the data of $X$. If $\modd_{p^*}\Gamma^*= \infty$, then ${\modd_p\Gamma=0}$.
\end{theorem}

Note that the conclusions in Theorems \ref{thm:lowerbound} and \ref{thm:upperbound} are biLipschitz invariant. Also recall that a complete metric space supporting a Poincar\'e inequality is $C$-quasiconvex for some $C=C(X)$. \emph{Thus we may, and will, assume that $X$ is a geodesic metric space}. Note that in geodesic spaces we can choose $\lambda_P=1$. For these facts see Theorem 8.3.2 and Remark 9.1.19 in \cite{HKST}. 

Theorem \ref{thm:duality} follows by combining Theorems \ref{thm:lowerbound} and \ref{thm:upperbound}, and recalling that $\mathcal{H}$ is comparable to the $(Q-1)$-dimensional Hausdorff 
measure in Ahlfors $Q$-regular spaces. Theorems \ref{thm:lowerbound} and \ref{thm:upperbound} will be proved in Sections \ref{section:lowerbound} and \ref{section:upperbound}, respectively. We now show how they imply Corollary \ref{cor:intro}. 
\begin{proof}[Proof of Corollary \ref{cor:intro}]
The ``only if" part follows directly from Theorem \ref{thm:duality}. To prove the ``if" part, notice first that Theorem \ref{thm:duality} shows that \eqref{eq:pathintro} holds 
for all path families $\Gamma(E,F;G)$ joining continua $E$ and $F$ inside $G$. Injecting this estimate into the proof of Theorem 4.7 in \cite{HeinonenKoskela1998} 
shows that $f$ is locally quasisymmetric, with constants depending only on the given data. On the other hand, Theorem 10.9 of \cite{Tyson2001} shows that locally quasisymmetric maps satisfy \eqref{eq:pathintro} for all path families. The required linear local connectedness and Loewner properties of $X$ and $Y$ are guaranteed by \cite[Theorem 3.3]{Korte2007} and \cite[Theorem 5.7]{HeinonenKoskela1998}. The ``if" part follows. 
\end{proof}

\section{Proof of Theorem \ref{thm:lowerbound}}\label{section:lowerbound}
Let $X, G, E, F, \Gamma$ and $\Gamma^*$ be as in Section \ref{section:results}. Fix $1<p<\infty$. Note that the constant function $1/\dist(E, F)$ restricted on $G$ is admissible for $\Gamma$. Therefore $\modd_p\Gamma$ is finite.

We need the following result of Kallunki and Shanmugalingam \cite{KallunkiShan2001}: The locally Lipschitz $p$-capacity of $\Gamma$ is defined to be
\[
\capp_p^L\Gamma=\inf_\rho\int_G\rho^p\ d\mu, 
\]
where the infimum is taken over every non-negative Borel-measurable function $\rho$ that is an upper gradient to some locally Lipschitz function $u: G\rightarrow [0, 1]$ with $u|_E=0$ and $u|_F=1$. 

Theorem 1.1 in \cite{KallunkiShan2001} reads as follows: if $1<p<\infty$, then 
\begin{equation} \label{eq:caponmod}
\modd_p\Gamma=\capp_p^L\Gamma
\end{equation}
for any choice of $E, F$ and $G$.

The proof of Theorem \ref{thm:lowerbound} is based on the following coarea estimate. 
\begin{proposition}\label{prop:coarea}
Let $u: G\rightarrow \R$ be locally Lipschitz and let $\rho$ be a $p$-integrable upper gradient of $u$ in $G$. Let $g: G\rightarrow [0, \infty]$ be a $p^*$-integrable Borel function. Then
\begin{equation}\label{eq:coarea}
\int_{\R}^*\int_{u^{-1}(t)}g\ d\Ha dt\leq C\int_{G}g\rho\ d\mu
\end{equation}
for some $C=C(X)$.
\end{proposition}

Before proving Proposition \ref{prop:coarea}, we show how it together with \eqref{eq:caponmod} yields Theorem \ref{thm:lowerbound}. 

\begin{proof}[Proof of Theorem \ref{thm:lowerbound}]
First assume that $\modd_{p}\Gamma>0$. If $\modd_{p^*}\Gamma^*=\infty$, there is nothing to prove. Otherwise let $g\in L^{p^*}(G)$ be admissible for $\Gamma^*$. Let $u: G\rightarrow [0, 1]$ be locally Lipschitz with $u|_E=0$ and $u|_F=1$. Let $\rho$ be an upper gradient of $u$. We may assume that $\rho$ is $p$-integrable. Note that for every $t\in (0, 1)$ the set $u^{-1}(t)$ separates $E$ and $F$, and is closed in $G$. Moreover, by (\ref{eq:coarea}) $\Ha(u^{-1}(t))<\infty$ for almost every $t$. Proposition \ref{prop:coarea} and H\"older's inequality give
\[
1\leq \int_{(0, 1)}^*\int_{u^{-1}(t)}g\ d\Ha dt\leq C\int_Gg\rho\ d\mu\leq C\left(\int_Gg^{p^*}d\mu \right)^\frac{1}{p^*}\left(\int_G\rho^{p}\,d\mu \right)^\frac{1}{p}.
\]
Now take infima over admissible functions $g$ and $\rho$ and apply \eqref{eq:caponmod} to get the lower bound. The same argument leads to a contradiction if $\modd_{p^*}\Gamma^*$ is finite and $\modd_p\Gamma=0$.
\end{proof}

We start the proof of Proposition \ref{prop:coarea} with a classical estimate for Lipschitz functions. See \cite[Theorem 7.7]{Mattila} for a euclidean version.
\begin{lemma}\label{lemma:coarea1}
Let $u: G\rightarrow\R$ be $L$-Lipschitz and let $A$ be a $\mu$-measurable subset of $G$. Then 
\begin{equation}\label{eq:coarea1}
\int_\R^* \Ha(u^{-1}(t)\cap A)\ dt\leq C(X)L\mu(A).
\end{equation}
\end{lemma}
\begin{proof}
Since $\mu$ is a Radon-measure, we may assume that $A$ is open. Let $\delta>0$. Apply the $5r$-covering theorem to find a countable collection of disjoint balls $\{B_i\}$ with $B_i=B(x_i, r_i)\subset A$, $5r_i\leq \delta$ and
\[
A\subset\bigcup_i5B_i.
\]
Define a Borel function $g: \R\rightarrow [0, \infty]$ with
\[
g=\sum_i\frac{\mu(5B_i)}{5r_i}\chi_{u(5B_i)}.
\]
Now for every $t\in\R$ we have $\Ha_\delta(u^{-1}(t)\cap A)\leq g(t)$, so by the doubling property of $\mu$, 
\begin{align*}
\int_\R^* \Ha_\delta(u^{-1}(t)\cap A)\ dt&\leq \int_\R g(t)\ dt \\
&\leq \sum_i\frac{\mu(5B_i)}{5r_i}|u(5B_i)| \\
&\leq C(X)L\sum_i\mu(B_i) \\
&\leq C(X)L\mu(A).
\end{align*}
Applying the monotone convergence theorem for upper integrals finishes the proof.
\end{proof}
The Poincar\'e inequality comes into play with the following lemma.

\begin{lemma}\label{lemma:newtonapprox}
Let $U\subset G$ be open and connected and suppose $v: U\rightarrow \R$ is locally integrable and $\rho_v: X\rightarrow [0, \infty]$ is an upper gradient of $v$ in $U$ that vanishes outside $G$. Let $N\subset U$ be the set of Lebesgue points of $v$. Then
\[
|v(x)-v(y)|\leq C(X)|x-y|(\mathcal{M}_{10|x-y|}\rho_v(x)+\mathcal{M}_{10|x-y|}\rho_v(y))
\]
whenever $x, y\in B\cap N$ for some ball $B$ that satisfies $5B\subset U$.
\end{lemma}
\begin{proof}
The case $U=X$ is classical and proved in, for example, \cite[Theorem 8.1.7]{HKST}. We follow the same proof for the case of general $U$. Let $B=B(x_0, r)$ satisfy $5B\subset U$. Let $x\in B$ be a Lebesgue point of $v$. The first part of the proof of \cite[Theorem 8.1.7]{HKST} shows that 
\begin{equation}\label{eq:estimaatti}
|v(x)-v_{B}|\leq Cr\mathcal{M}_{4r}\rho_v(x)
\end{equation}
for some constant $C=C(X)$. Let $y$ be another Lebesgue point of $v$ in $B$. If $r\leq \frac{5}{2}|x-y|$, then applying (\ref{eq:estimaatti}) twice gives the desired result. Otherwise apply (\ref{eq:estimaatti}) with $B(x, 2|x-y|)$ instead.
\end{proof}
\begin{proof}[Proof of Proposition \ref{prop:coarea}]
By standard real analysis arguments it suffices to show that 
\begin{equation}\label{eq:coarea1}
\int_{[0, 1]}^* \Ha(u^{-1}(t)\cap A)\ dt\leq C(X)\int_{A}\rho\ d\mu.
\end{equation}
Let us first show that
\[
\int_{[0, 1]}^* \Ha(u^{-1}(t)\cap A\cap B)\ dt\leq C(X)\int_{A\cap B}\mathcal{M}_{10\diam B}\rho\ d\mu
\]
for any Borel set $A\subset G$ and any ball $B\subset 5B\subset G$. Continuity of $u$ and Lemma \ref{lemma:newtonapprox} give 
\begin{equation}\label{eq:ulip}
|u(x)-u(y)|\leq C(X)|x-y|(\mathcal{M}_{10\diam B}\rho(x)+\mathcal{M}_{10\diam B}\rho(y))
\end{equation}
for any $x, y\in B$. 

Let $B_k=\{x\in B\ |\ 2^k<\mathcal{M}_{10\diam B}\rho(x)\leq 2^{k+1}\}$. Abuse the notation and define the sets $B_{-\infty}$ and $B_\infty$ as the sets of points $x\in B$ where, respectively, $\mathcal{M}_{10\diam B}\rho(x)=0$ or $\mathcal{M}_{10\diam B}\rho(x)=\infty$. Recall that we assume $u$ to be locally Lipschitz. Since $B$ is compactly contained in $G$, $u|_B$ is Lipschitz. Now Lemma \ref{lemma:coarea1} applied to any Lipschitz extension of $u|_B$ implies that
\[
\int_{[0, 1]}^* \Ha(u^{-1}(t)\cap A\cap B_\infty)\ dt = 0, 
\]
since the integrability of $\mathcal{M}\rho$ implies that $\mu(B_\infty)=0$. On the other hand, if $B_{-\infty} \neq \emptyset$ then $\rho=0$ almost everywhere in $5B$. Since we may assume that $X$ is geodesic, 
it moreover follows that $u$ is constant in $B$. We conclude that $\Ha(u^{-1}(t)\cap A\cap B_{-\infty})$ is nonzero for at most one $t$.

It follows from (\ref{eq:ulip}) that $u|_{B_k}$ is $C(X)2^{k}$-Lipschitz. Let $u_k: X\rightarrow \R$ be any Lipschitz extension of $u|_{B_k}$ with the same Lipschitz constant. Now the previous observations together with the monotone convergence theorem, Lemma \ref{lemma:coarea1} and the definition of $B_k$ give
\begin{align*}
\int_{[0, 1]}^*\Ha(u^{-1}(t)\cap A\cap B)\ dt&=\sum_k\int_{[0, 1]}^*\Ha(u_k^{-1}(t)\cap A\cap B_k)\ dt \\
&\leq C(X)\sum_k2^k\mu(A\cap B_k) \\
&\leq C(X)\int_{A\cap B}\mathcal{M}_{10\diam B}\rho\ d\mu.
\end{align*}
Applying a Whitney-type covering, see Lemma \ref{lemma:whitneypeite}, we get 
\[
\int_{[0, 1]}^* \Ha(u^{-1}(t)\cap A)\ dt\leq C(X)\int_{A}\mathcal{M}_{10R}\rho\ d\mu,
\]
where $R$ is the supremum of the diameters of the balls used in the cover. We can make $R$ arbitrarily small, as is implied by Lemma \ref{lemma:whitneypeite}. The Lebesgue differentiation theorem 
and dominated convergence then yield (\ref{eq:coarea1}).
\end{proof}

\section{Proof of Theorem \ref{thm:upperbound}}\label{section:upperbound}
Consider the sets 
\begin{equation}\label{eq:pullistuma}
\Gamma^*_j=\{ S\in\Gamma^*\ |\ \dist(S, E\cup F)>j^{-1}\}.
\end{equation}
By applying the proof of Proposition 5.2.11 in \cite{HKST} and the general Fuglede's lemma, see \cite[Theorem 3]{Fuglede1957}, it can be shown that
\begin{equation}\label{eq:pullistumalimit}
\lim_{j\rightarrow\infty}\modd_{p^*}\Gamma^*_j=\modd_{p^*}\Gamma^*.
\end{equation}

The following result is the key tool in connecting the two moduli. 
\begin{lemma}(Relative isoperimetric inequality)\label{lemma:ipm}\\
Let $S\in\Gamma^*$ and let $U$ be the component of $G-S$ that contains $E$. There are constants $C=C(X)$ and $\lambda=\lambda(X)>1$ such that 
\[
\mathrm{min}\left\lbrace \frac{\mu(B-U)}{\mu(B)}, \frac{\mu(B\cap U)}{\mu(B)}\right\rbrace\leq C\frac{r}{\mu(\lambda B)}\Ha(\partial U\cap \lambda B)
\]
for all balls $B\subset\subset G$.
\end{lemma}
\begin{proof}
Given a ball $B\subset\subset G$ there is a larger ball $B'\subset G$ with $B\subset B'$ and $\Ha(\partial B')<\infty$ (apply Lemma \ref{lemma:coarea1} below to the distance function). Applying Theorem 6.2 of \cite{KorteLahti2014} shows that $B'\cap U$ is a so called \emph{set of finite perimeter}. The relative isoperimetric inequality for sets of finite perimeter follows from the $1$-Poincar\'e inequality by \cite[Theorem 1.1]{KorteLahti2014}. 
\end{proof}
Note that Lemma \ref{lemma:ipm} requires the weak $1$-Poincar\'e inequality. See Section \ref{section:examples} for examples of spaces that support a weak $(1+\varepsilon)$-Poincar\'e inequality for a given $\varepsilon>0$, but no relative isoperimetric inequality.

Fix $\gamma\in\Gamma$. The idea behind the proof of Theorem \ref{thm:upperbound} is to construct admissible functions $\phi^n_j$ of $\Gamma_j^*$ that are supported close to $|\gamma|$, and then apply Lemma \ref{lemma:variation} below. 

Let $n\geq 2$ be a natural number and let $\mathcal{B}^n$ be the collection of balls obtained by applying Lemma \ref{lemma:whitneypeite} with $\Omega=G'$ and $A=|\gamma|\cap G'$. Moreover, given 
$k \in \mathbb{Z}$ let $\mathcal{G}^n_k=\mathcal{G}_k$ be the collections of balls constructed in the proof of \ref{lemma:whitneypeite}.  

Now let $S\in\Gamma^*$. Let $U$ be the component of $G-S$ that contains $E$.  Let
\[
T_n=\sup\left\lbrace t\in (0, 1)\ \bigg|\  \frac{\mu(U\cap B)}{\mu(B)}\geq\frac{1}{2}\text{ for all }B\in\mathcal{B}^n\text{ such that }\gamma(t)\in B\right\rbrace.
\]
Note that there exists an $\varepsilon>0$, so that $N_\varepsilon(E)\subset U$ and $N_\varepsilon(F)\subset G-U$. Combining this observation with Lemma \ref{lemma:whitneypeite} (i) and continuity of $\gamma$ shows that $T_n$ is well defined and that $T_n\in (0, 1)$. It follows that there exist balls $B_i=B(x_i, r_i)\in \mathcal{B}^n$ for $i=1, 2$ such that $B_1\cap B_2\neq\emptyset$ and
\[
\frac{\mu(B_1\cap U)}{\mu(B_1)}\leq \frac{1}{2}\leq \frac{\mu(B_2\cap U)}{\mu(B_2)}.
\]
Now let $x\in B_1\cap B_2$ and let $i\in\{1, 2\}$ be the index for which $r_i=\max\{r_1, r_2\}$. Let $B=B(x, 2r_i)$. It follows from Lemma \ref{lemma:whitneypeite} (ii) and Lemma \ref{lemma:ipm}, that 
\[
C(X)\leq\mathrm{min}\left\lbrace \frac{\mu(B-U)}{\mu(B)}, \frac{\mu(B\cap U)}{\mu(B)}\right\rbrace\leq C'\frac{r_i}{\mu(\lambda B)}\Ha(\partial U\cap \lambda B) 
\]
for some $C'=C'(X)$ and $\lambda=\lambda(X)$. Therefore
\[
\Ha(S\cap \lambda'B_i)\geq\Ha(\partial U\cap\lambda'B_i)\geq \frac{1}{C(X)}r_i^{-1}\mu(B_i),
\]
where $\lambda'=1+2\lambda$. We conclude that the function
\[
\phi^n=C\sum_{B\in\mathcal{B}^n}r_{B}\mu(B)^{-1}\chi_{\lambda'B},
\]
where $r_B$ is the radius of $B$, is admissible for $\Gamma^*$, but it may not be $p^*$-integrable. This is why we consider the families $\Gamma_j^*$ instead. 

Note that if $5B\in\mathcal{B}^n$ satisfies $B\in \mathcal{G}^n_{-k}$ for sufficiently large $k$ depending on $j$ and $n$, then given any $S\in \Gamma_j^*$  
\[
\frac{\mu(U\cap B)}{\mu(B)}\in \{0, 1\}.
\]
Here $U$ is again the component of $G-S$ that contains $E$. Together with the construction of $\phi^n$ this implies that there is a $k(j, n)\in\Z$ such that 
\[
\phi^n_j=C\sum_{k\geq k(j, n)}\sum_{B: \frac{1}{5}B\in \mathcal{G}^n_k}r_B\mu(B)^{-1}\chi_{\lambda'B}
\]
is admissible for $\Gamma^*_j$. It is $p^*$-integrable, since each $\mathcal{G}^n_k$ contains only finitely many balls and $G$ is bounded.

Now let $j$ be large enough, so that $\modd_{p^*}\Gamma^*_j$ is nonzero. The existence of such a $j$ follows by combining Theorem \ref{thm:lowerbound} with \eqref{eq:pullistumalimit}. Apply Lemma \ref{lemma:minimizer} to find a $p^*$-weakly admissible  function $\rho_j$ of $\Gamma_j^*$ with the property
\[
\modd_{p^*}\Gamma_j^*=\int_G\rho_j^{p^*}\ d\mu.
\]
\begin{lemma}\label{lemma:variation}
Let $\phi$ be another $p^*$-integrable, $p^*$-weakly admissible function of $\Gamma_j^*$. Then
\[
\modd_{p^*}\Gamma_j^*\leq \int_G \phi\rho_j^{p^*-1}\ d\mu.
\]
\end{lemma}
\begin{proof}
For any $t\in [0, 1]$ let $\omega_t=t\phi+(1-t)\rho_j$. Now for any $t$
\[
\modd_{p^*}\Gamma_j^*\leq \int_G \omega_t^{p^*} d\mu
\]
with equality at $t=0$. It follows that 
\[
0\leq \frac{d}{dt}\bigg|_{t=0}\int_G\omega_t^{p^*} d\mu=p^*\int_G(\phi-\rho_j)\rho_j^{p^*-1} \ d\mu,
\]
which finishes the proof.
\end{proof}

Applying Lemma \ref{lemma:variation}, the doubling property of $\mu$, the definition of the Hardy-Littlewood maximal operator and (iii) gives 
\begin{align*}
\modd_{p^*}\Gamma_j^*&\leq \int_G\phi_j^n\rho_j^{p^*-1}\ d\mu \\
&\leq C(X)\sum_{B\in\mathcal{B}^n}r_B\dashint_{\lambda'B}\rho_j^{p^*-1}\ d\mu \\
&\leq C(X)\sum_{B\in\mathcal{B}^n}r_B\inf_{x\in B}\mathcal{M}_{C(X, G)/n}(\rho_j^{p^*-1})(x)\\
&\leq C(X)\int_{|\gamma|}\mathcal{M}_{C(X, G)/n}(\rho_j^{p^*-1})\ d\Ha^1.
\end{align*}
Letting $n\rightarrow\infty$ and applying Fuglede's lemma \cite[p. 131]{HKST} we see that $C(\modd_{p^*}\Gamma^*_j)^{-1}\rho^{p^*-1}_j$ is admissible for $\Gamma$. Therefore
\[
(\modd_p\Gamma)^\frac{1}{p}\leq C(\modd_{p^*}\Gamma^*_j)^{-1}\left(\int_G\rho_j^{p^*}\ d\mu\right)^\frac{1}{p}=C(\modd_{p^*}\Gamma^*_j)^{-\frac{1}{p^*}}.
\]
Applying \eqref{eq:pullistumalimit} finishes the proof.
\section{Counter-examples}\label{section:examples}
The relative isoperimetric inequality is an instrumental part of the proof of Theorem \ref{thm:upperbound}. By \cite{KorteLahti2014} it is equivalent to the weak $1$-Poincar\'e inequality. Let $\varepsilon\in (0, 1)$. We now construct a space $X$ that satisfies the hypotheses of Theorem \ref{thm:duality} apart from the 1-Poincar\'e inequality. Instead, $X$ will support a $(1+\varepsilon)$-Poincar\'e inequality. 

Let $K\subset [1/4, 3/4]$ be a self-similar Cantor set with Hausdorff-dimension $1-\varepsilon$ and the following property: for all $x\in K$ and $0<r<1$ 
\[
\Ha^{1-\varepsilon}_\infty(K\cap B(x, r))\geq Cr^{1-\varepsilon}
\]
for some $C>0$ that does not depend on $r$. Let $Q=[0, 1]^3\subset\R^3$ and let $A=[1/4, 3/4]\times K\times\{0\}\subset Q$. Then for any $x\in A$ and $0<r\leq \mathrm{diam}(Q)$ 
\begin{equation}\label{eq:poincareehto}
\Ha^{2-\varepsilon}_\infty(A\cap B(x, r))\geq Cr^{2-\varepsilon}
\end{equation}
for some (other) $C>0$ that does not depend on $r$. Let $Q_1$ and $Q_2$ be two copies of the space $Q$. Finally, let $X=Q_1\sqcup_A Q_2$, two cubes glued together along $A$. Equip $X$ with the geodesic metric $d$ that restricts to the metrics of the cubes in either cube, and for $x\in Q_1$ and $y\in Q_2$ set 
\[
d(x, y)=\inf_{a\in A}(|x-a|+|a-y|).
\] 
Equip $X$ with the measure $\mu$ that restricts to the $3$-dimensional Lebesgue measure on both cubes. It follows immediately from the definitions that $(X, d, \mu)$ is a complete geodesic Ahlfors $3$-regular metric space. The validity of a weak $(1+\varepsilon)$-Poincar\'e inequality follows from (\ref{eq:poincareehto}) and \cite[Theorem 6.15]{HeinonenKoskela1998}. 

Now let $E\subset Q_1-A$ and $F\subset Q_2-A$ be nondegenerate continua and let $G=X$. Let $\Gamma$ and $\Gamma^*$ be as in Theorem \ref{thm:duality}. The modulus $\modd_3\Gamma$ is non-zero and finite, since $X$ is Loewner, see \cite{HeinonenKoskela1998}. On the other hand $\modd_{3^*}\Gamma^*=\infty$, since $\Gamma^*$ does not admit any admissible functions. To see this, note that $A$ separates $E$ and $F$ in $G$, but has vanishing $2$-measure. We conclude that $X$ does not satisfy the upper bound of Theorem \ref{thm:duality}. Note that this implies that $X$ does not support a weak $1$-Poincar\'e inequality. This can also be deduced from the main result of \cite{KorteLahti2014}.

\bibliographystyle{plain}
\bibliography{modulijuttubib}

\begin{thebibliography}{10}

\bibitem{GMT}
Herbert Federer.
\newblock {\em Geometric measure theory}.
\newblock Die Grundlehren der mathematischen Wissenschaften, Band 153.
  Springer-Verlag New York Inc., New York, 1969.

\bibitem{Fuglede1957}
Bent Fuglede.
\newblock Extremal length and functional completion.
\newblock {\em Acta Math.}, 98:171--219, 1957.

\bibitem{Gehring1962}
F.~W. Gehring.
\newblock Extremal length definitions for the conformal capacity of rings in
  space.
\newblock {\em Michigan Math. J.}, 9:137--150, 1962.

\bibitem{HeinonenKoskela1998}
Juha Heinonen and Pekka Koskela.
\newblock Quasiconformal maps in metric spaces with controlled geometry.
\newblock {\em Acta Math.}, 181(1):1--61, 1998.

\bibitem{HKST}
Juha Heinonen, Pekka Koskela, Nageswari Shanmugalingam, and Jeremy~T. Tyson.
\newblock {\em Sobolev spaces on metric measure spaces, an approach based on
  upper gradients}, volume~27 of {\em New Mathematical Monographs}.
\newblock Cambridge University Press, Cambridge, 2015.

\bibitem{JonesLahtiShan2018}
Rebekah Jones, Panu Lahti, and Nageswari Shanmugalingam.
\newblock Modulus of families of sets of finite perimeter and quasiconformal
  maps between metric spaces of globally {Q}-bounded geometry, preprint,
  arxiv:1806.06211.

\bibitem{KallunkiShan2001}
Sari Kallunki and Nageswari Shanmugalingam.
\newblock Modulus and continuous capacity.
\newblock {\em Ann. Acad. Sci. Fenn. Math.}, 26(2):455--464, 2001.

\bibitem{Korte2007}
Riikka Korte.
\newblock Geometric implications of the {P}oincar\'e inequality.
\newblock {\em Results Math.}, 50(1-2):93--107, 2007.

\bibitem{KorteLahti2014}
Riikka Korte and Panu Lahti.
\newblock Relative isoperimetric inequalities and sufficient conditions for
  finite perimeter on metric spaces.
\newblock {\em Ann. Inst. H. Poincar\'e Anal. Non Lin\'eaire}, 31(1):129--154,
  2014.

\bibitem{LehtoVirtanenQC}
O.~Lehto and K.~I. Virtanen.
\newblock {\em Quasiconformal mappings in the plane}.
\newblock Springer-Verlag, New York-Heidelberg, second edition, 1973.
\newblock Translated from the German by K. W. Lucas, Die Grundlehren der
  mathematischen Wissenschaften, Band 126.

\bibitem{Mattila}
Pertti Mattila.
\newblock {\em Geometry of sets and measures in {E}uclidean spaces}, volume~44
  of {\em Cambridge Studies in Advanced Mathematics}.
\newblock Cambridge University Press, Cambridge, 1995.
\newblock Fractals and rectifiability.

\bibitem{RickmanQR}
Seppo Rickman.
\newblock {\em Quasiregular mappings}, volume~26 of {\em Ergebnisse der
  Mathematik und ihrer Grenzgebiete (3) [Results in Mathematics and Related
  Areas (3)]}.
\newblock Springer-Verlag, Berlin, 1993.

\bibitem{Schwartz1973}
Laurent Schwartz.
\newblock {\em Radon measures on arbitrary topological spaces and cylindrical
  measures}.
\newblock Published for the Tata Institute of Fundamental Research, Bombay by
  Oxford University Press, London, 1973.
\newblock Tata Institute of Fundamental Research Studies in Mathematics, No. 6.

\bibitem{Semmes1996b}
S.~Semmes.
\newblock Finding curves on general spaces through quantitative topology, with
  applications to {S}obolev and {P}oincar\'{e} inequalities.
\newblock {\em Selecta Math. (N.S.)}, 2(2):155--295, 1996.

\bibitem{Tyson2001}
Jeremy~T. Tyson.
\newblock Metric and geometric quasiconformality in {A}hlfors regular {L}oewner
  spaces.
\newblock {\em Conform. Geom. Dyn.}, 5:21--73, 2001.

\bibitem{Vaisala}
Jussi V\"ais\"al\"a.
\newblock {\em Lectures on {$n$}-dimensional quasiconformal mappings}.
\newblock Lecture Notes in Mathematics, Vol. 229. Springer-Verlag, Berlin-New
  York, 1971.

\bibitem{Ziemer1967}
William~P. Ziemer.
\newblock Extremal length and conformal capacity.
\newblock {\em Trans. Amer. Math. Soc.}, 126:460--473, 1967.

\end{thebibliography}
\vspace{1em}
\noindent
Department of Mathematics and Statistics, University of Jyv\"askyl\"a, P.O.
Box 35 (MaD), FI-40014, University of Jyv\"askyl\"a, Finland.\\

\emph{E-mail:} \settowidth{\hangindent}{\emph{aaaaaaaaa}}A.L.: \textbf{atte.s.lohvansuu@jyu.fi} \\ K.R.: \textbf{kai.i.rajala@jyu.fi}
\end{document}